\documentclass[[a4paper,UKenglish,cleveref, autoref, thm-restate]{lipics-v2021}
\usepackage{extarrows}
\usepackage{graphicx}
\usepackage{mathtools}
\usepackage[ruled]{algorithm2e}
\usepackage{wrapfig}

\allowdisplaybreaks
\nolinenumbers

\newcounter{ProblemCounter}

\newtheorem{lem}[theorem]{Lemma}
\newtheorem{prop}[theorem]{Proposition}

\theoremstyle{definition}

\theoremstyle{definition}

\newcommand{\mun}{\mathfrak{u}(n)}
\newcommand{\Z}{\mathbb{Z}}
\newcommand{\N}{\mathbb{N}}
\newcommand{\Q}{\mathbb{Q}}

\newcommand{\K}{\mathbb{K}}

\newcommand{\Qp}{\mathbb{Q}_{\geq 0}}
\newcommand{\Qpp}{\mathbb{Q}_{>0}}
\newcommand{\Zp}{\mathbb{Z}_{\geq 0}}
\newcommand{\Zpp}{\mathbb{Z}_{>0}}
\newcommand{\UT}{\mathsf{UT}}

\newcommand{\HH}{\operatorname{H}}
\newcommand{\PI}{\operatorname{PI}}

\newcommand{\supp}{\operatorname{supp}}
\newcommand{\card}{\operatorname{card}}

\newcommand{\mG}{\mathcal{G}}

\newcommand{\mL}{\mathcal{L}}

\newcommand{\mC}{\mathcal{C}}

\newcommand{\mS}{\mathcal{S}}
\newcommand{\mH}{\mathcal{H}}

\newcommand{\bl}{\boldsymbol{\ell}}

\newcommand{\ba}{\boldsymbol{a}}
\newcommand{\bb}{\boldsymbol{b}}

\newcommand{\bv}{\boldsymbol{v}}
\newcommand{\bx}{\boldsymbol{x}}
\newcommand{\by}{\boldsymbol{y}}
\newcommand{\bn}{\boldsymbol{n}}
\newcommand{\bzero}{\boldsymbol{0}}

\begin{document}

\title{Semigroup intersection problems in the Heisenberg groups}

\author{Ruiwen Dong}{Department of Computer Science, University of Oxford}{ruiwen.dong@kellogg.ox.ac.uk}{}{}

\authorrunning{R. Dong}

\Copyright{Ruiwen Dong} 



\category{} 


\supplement{}

\maketitle

\begin{abstract}
We consider two algorithmic problems concerning sub-semigroups of Heisenberg groups and, more generally, two-step nilpotent groups.
The first problem is \emph{Intersection Emptiness}, which asks whether a finite number of given finitely generated semigroups have empty intersection.
This problem was first studied by Markov in the 1940s.
We show that Intersection Emptiness is PTIME decidable in the Heisenberg groups $\operatorname{H}_{n}(\mathbb{K})$ over any algebraic number field $\mathbb{K}$, as well as in direct products of Heisenberg groups.
We also extend our decidability result to arbitrary finitely generated 2-step nilpotent groups.

The second problem is \emph{Orbit Intersection}, which asks whether the orbits of two matrices under multiplication by two semigroups intersect with each other.
This problem was first studied by Babai \emph{et al}.\ (1996), who showed its decidability within commutative matrix groups.
We show that Orbit Intersection is decidable within the Heisenberg group $\operatorname{H}_{3}(\mathbb{Q})$.
\end{abstract}

\newpage

\section{Introduction}
The computational theory of matrix groups and semigroups is one of the
oldest and most well-developed parts of computational algebra.
Dating back to the work of Markov~\cite{markov1947certain} in the 1940s, the area plays an essential role in analysing system dynamics, with notable applications in automata theory and program analysis~\cite{blondel2005decidable, choffrut2005some, derksen2005quantum, hrushovski2018polynomial}.
While many computational problems are undecidable even for matrix groups of dimension three and four~\cite{bell2010undecidability, mikhailova1966occurrence, paterson1970unsolvability}, various non-trivial algorithms have been developed for matrix groups satisfying additional constraints, such as commutativity~\cite{babai1996multiplicative}, nilpotency~\cite{https://doi.org/10.48550/arxiv.2208.02164}, solvability~\cite{kopytov1968solvability}, and having dimension two~\cite{bell2017identity, potapov2017decidability}.

As most algorithmic problems for commutative groups are well-understood due to their relatively simple structure, much effort has focused on problems concerning relaxations of the commutativity requirement, such as nilpotency and solvability.
Prominent examples of widely studied groups include the \emph{Heisenberg groups}, as well as the more general \emph{2-step nilpotent groups}.
The Heisenberg groups $\HH_n(\K)$ play an important role in many branches of mathematics, physics and computer science.
They first arose in the description of one-dimensional quantum mechanical systems~\cite{Neumann1931, weyl1950theory}, and have now become an important mathematical object connecting domains like representation theory, theta functions, Fourier analysis and quantum algorithms~\cite{howe1980role, igusa2012theta, kirillov2004lectures, krovi2008efficient, yang1991harmonic}.
From a computational point of view, Heisenberg groups are interesting because they are the simplest non-commutative Lie groups.
Heisenberg groups are included in the class of 2-step nilpotent groups: these are groups whose quotient by their centre is abelian.
Despite being the simplest class of non-commutative groups, 2-step nilpotent groups admit highly non-trivial or even undecidable algorithmic problems, notably due to their ability to encode quadratic equations~\cite{konig2016knapsack}.
For example, decades of research has focused on finding a polynomial-time group isomorphism algorithm for 2-step nilpotent groups, with little success~\cite{babai2011code, GARZON1991237}.

For a set $\mG$ of matrices in some matrix group $G$, denote by $\langle \mG \rangle$ the semigroup generated by the set $\mG$.
In this paper, we consider the following two decision problems for the Heisenberg groups and 2-step nilpotent groups.
\begin{enumerate}[i.]
    \item \textit{(Intersection Emptiness)} Given $M$ sets of matrices $\mG_1,\ldots, \mG_M$, decide whether \\
    $\langle \mG_1 \rangle \cap \cdots \cap \langle \mG_M \rangle = \emptyset$.
    \item \textit{(Orbit Intersection)} Given two sets of matrices $\mG, \mH$ and matrices $S, T$, decide whether \\
    $T \cdot \langle \mG \rangle \cap S \cdot \langle \mH \rangle = \emptyset$.
    \setcounter{ProblemCounter}{\value{enumi}}
\end{enumerate}

Intersection Emptiness was one of the first problems studied in algorithmic semigroup theory.
In the seminal work of Markov~\cite{markov1947certain}, the undecidability of Intersection Emptiness was shown for two sets of $4 \times 4$ integer matrices.
More recently, by encoding the \emph{Post Correspondence Problem}, Halava and Harju showed its undecidability for two sets of $3 \times 3$ upper triangular integer matrices~\cite{halava2007markov}.
For $2 \times 2$ integer matrices, the problem is only known to be NP-hard~\cite{bell2012computational}.
In this paper, we show that Intersection Emptiness is decidable in polynomial time for the Heisenberg groups $\HH_{n}(\K)$ over an arbitrary algebraic number field $\K$, as well as for any direct product of such Heisenberg groups.
In fact, we will prove the decidability result in the more general case of (finitely generated) 2-step nilpotent groups.

The Orbit Intersection problem was first considered by Babai \emph{et al.}~\cite{babai1996multiplicative}, who proved its decidability in \emph{commutative} matrix groups over an algebraic number field.
In this paper, we prove the decidability of Orbit Intersection for matrices in the Heisenberg group $\HH_{3}(\Q)$.

Let us mention some previous work for semigroup algorithmic problems in the Heisenberg groups and 2-step nilpotent groups.
These have seen significant advance in research in recent years. 
Various results have been shown for the following decision problems.
\begin{enumerate}[i.]
    \setcounter{enumi}{\value{ProblemCounter}}
    \item \textit{(Identity Problem)} Given a set of matrices $\mG$, decide whether the identity matrix $I \in \langle \mG \rangle$.
    \item \textit{(Membership Problem)} Given a set of matrices $\mG$ and a matrix $A$, decide whether $A \in \langle \mG \rangle$.
    \item \textit{(Knapsack Problem)} Given matrices $A_1, A_2, \ldots, A_K$ and a matrix $A$, decide whether there exist $(n_1, n_2 \ldots, n_K) \in \N^K$ such that $A = A_1^{n_1} A_2^{n_2} \cdots A_K^{n_K}$.
    \setcounter{ProblemCounter}{\value{enumi}}
\end{enumerate}

The Identity Problem in $\HH_{n}(\Q)$ was shown to be decidable by Ko, Niskanen and Potapov~\cite{ko2017identity}.
Dong~\cite{https://doi.org/10.48550/arxiv.2208.02164} then introduced tools from Lie algebra and strengthened this result to PTIME decidability in $\HH_{n}(\K)$ for algebraic number fields $\K$.
The Membership Problem in $\HH_n(\Q)$ was shown to be decidable by Colcombet, Ouaknine, Semukhin and Worrell.
Their main idea is to use the \emph{Baker-Campbell-Hausdorff (BCH) formula} as well as to incorporate the Membership Problem in a Parikh automaton.
It was left as an open problem whether the Membership Problem in $\HH_n(\K)$ for larger fields $\K$ remains decidable.
On the other hand, it is known that there exist 2-step nilpotent groups with undecidable Membership Problem~\cite{lefaucheux2022private}.
As for the Knapsack Problem, K{\"o}nig, Lohrey and Zetzsche showed its decidability in $\HH_n(\Z)$ by reducing it to solving a single quadratic equation over the natural numbers~\cite{konig2016knapsack}.
They also constructed a 2-step nilpotent group (namely, a direct product of $\HH_3(\Z)$) where the Knapsack Problem is undecidable, using an embedding of Hilbert's Tenth Problem.

We point out that by taking $\mG_1 = \mG$, $\mG_2 = \{I\}$, Intersection Emptiness subsumes the Identity Problem.
Whereas by taking $T = I, S = A, \mH = \{I\}$, the Orbit Intersection problem subsumes the Membership Problem.
Hence, the tools developed in this paper provide a more general approach to semigroup problems in 2-step nilpotent groups.
Our proofs are based on the logarithm of matrices and the BCH formula, whose usage in studying matrix semigroup problems has been introduced in~\cite{colcombet2019reachability} and \cite{https://doi.org/10.48550/arxiv.2208.02164}.
However, our approach goes much deeper in analysing the non-commutative terms of the BCH formula.
We show that these terms are connected with a word combinatorics problem concerning subwords of length two, and show a critical result characterizing the behaviour of these terms.
This will allow us to reduce equations containing word combinatorial terms to pure linear Diophantine equations.

\section{Main results}
In this section we state our main results.
Denote by $\UT(n, \Q)$ the group of $n \times n$ upper triangular rational matrices with ones along the diagonal.
Our main result on Intersection Emptiness is the following. For the formal definition of 2-step nilpotency, see Section~\ref{sec:prelim}.
\begin{restatable}{thrm}{thminterempty}\label{thm:interempty}
    Let $G$ be a 2-step nilpotent subgroup of $\UT(n, \Q)$ for some $n$. 
    Given finite subsets $\mG_1, \ldots, \mG_M$ of $G$, it is decidable in polynomial time whether $\langle \mG_1 \rangle \cap \cdots \cap \langle \mG_M \rangle = \emptyset$.
\end{restatable}

For $n \geq 3$, the \emph{Heisenberg group} $\HH_n(\K)$ over a field or commutative ring $\K$ is defined as
\[
\HH_n(\K) \coloneqq 
\left\{
\begin{pmatrix}
1 & \ba^{\top} & c \\
0 & I_{n-2} & \bb \\
0 & 0 & 1 \\
\end{pmatrix}
, \text{ where $\ba, \bb \in \K^{n-2}$, $c \in \K$ } \right\},
\]
where we use the notation $I_{d}$ for the identity matrix of dimension $d$.
Decidability results for the Heisenberg groups and for 2-step nilpotent groups follow as a corollary of Theorem~\ref{thm:interempty}.
\begin{restatable}{cor}{corheisnilp}\label{cor:heisnilp}
    Intersection Emptiness is decidable:
    \begin{enumerate}[(i)]
        \item in PTIME, for the Heisenberg groups $\HH_{n}(\K)$ over any algebraic number field $\K$, and for any direct product of Heisenberg groups.
        \item for finitely generated 2-step nilpotent groups\footnote{We suppose that the structure of the group is given by a finite presentation or a consistent polycyclic presentation (see~\cite[Chapter~8]{holt2005handbook}).}.
    \end{enumerate}
\end{restatable}

Fix a group $G$. 
Given an element $T \in G$ and a subset $\mG$ of $G$, denote by $T \cdot \langle \mG \rangle$ the orbit of $T$ under right multiplication by the semigroup $\langle \mG \rangle$.
That is,
$
T \cdot \langle \mG \rangle \coloneqq \{T \cdot s \mid s \in \langle \mG \rangle\}.
$
Our main result concerning Orbit Intersection is the following.
\begin{restatable}{thrm}{thmshiftinter}\label{thm:shiftinter}
    Given elements $T, S \in \HH_3(\Q)$ and two finite subsets $\mG, \mH$ of $\HH_3(\Q)$, it is decidable whether $T \cdot \langle \mG \rangle \cap S \cdot \langle \mH \rangle = \emptyset$.
\end{restatable}

\section{Preliminaries}\label{sec:prelim}
\paragraph*{Convex geometry}
Let $V$ be a $\Q$-linear space.
A subset $\mC \subseteq V$ is called a \emph{cone} if $a \in \mC$ implies $a\Qp \subseteq \mC$, and $a, b \in \mC$ implies $a + b \in \mC$.
Given a set of vectors $\mS \subseteq V$, denote by $\langle \mS \rangle_{\Qp}$ the \emph{cone generated by $\mS$},
that is, the smallest cone of $V$ containing $\mS$.
The \emph{dimension} of a cone $\mC$ is the dimension of the smallest linear space containing $\mC$.

The \emph{support} of a vector $\bl = (\ell_1, \ldots, \ell_K) \in \Zp^K$ is defined as the set of indices where the entry of $\bl$ is non-zero:
\[
\supp(\bl) \coloneqq \{i \in \{1, \ldots, K\} \mid \ell_i > 0\}.
\]
The \emph{support} of a subset $\Lambda$ of $\Zp^K$ is defined as the union of supports of all vectors in $\Lambda$:
\[
\supp(\Lambda) \coloneqq \bigcup_{\bl \in \Lambda} \supp(\bl) = \{i \mid \exists (\ell_1, \ldots, \ell_K) \in \Lambda, \ell_i > 0\}.
\]

In this paper, we will need to compute the support of sets of the form $\Lambda = \Zp^K \cap V$, where $V$ is a $\Q$-linear subspace of $\Q^K$.

\begin{lem}[{\cite[Lemma~2.4]{https://doi.org/10.48550/arxiv.2208.02164}}]\label{lem:compsupp}
Given $V$ a $\Q$-linear subspace of $\Q^K$, represented as the solution set of linear homogeneous equations, one can compute the support of $\Lambda = \Zp^K \cap V$ in polynomial time.
\end{lem}

\paragraph*{The group $\UT(n, \Q)$ and 2-step nilpotent groups}

Denote by $\UT(n, \Q)$ the group of $n \times n$ upper triangular rational matrices with ones along the diagonal.
Let $\K$ be an algebraic number field.
$\K$ can be considered as a linear space over $\Q$ of dimension $d \coloneqq [\K : \Q]$.
Let $k_1, \ldots, k_d \in \K$ be a $\Q$-basis of this linear space.
Throughout this paper, an element $k$ of $\K$ is represented as a tuple $(a_1, \ldots, a_d) \in \Q^d$ such that $k = a_1 k_1 + \cdots a_d k_d$.
An element $k$ of $\K$ acts on $\K$ by multiplication, and can therefore be considered as an endomorphism of the $\Q$-linear space $\K$.
Associate $k$ with the matrix that represents this endomorphism, then we have an (injective) embedding $\iota: \K \hookrightarrow \Q^{d \times d}$.
In particular, $\iota(1) = I_d$.
This embedding is effectively computable in polynomial time~\cite{cohen1993course}.

The embedding $\iota$ extends to an embedding $\HH_{n}(\K) \hookrightarrow \UT(nd, \Q)$, which we also denote by $\iota$. Note that for any matrix $A \in \HH_{n}(\K)$, the total bit size of entries in $\iota(A)$ is at most quadratic in the total bit size of entries in $A$.
Therefore, throughout this paper, we will work with matrices in $\iota(\HH_{n}(\K)) \subseteq \UT(nd, \Q)$, knowing that any polynomial time algorithm in $\UT(nd, \Q)$ will translate to a polynomial time algorithm in $\HH_{n}(\K)$.

Let $G$ be an arbitrary group.
The \emph{centre} of $G$ is the normal subgroup $Z(G) \trianglelefteq G$ consisting of elements that commute with every element of $G$ (see~\cite{dructu2018geometric}).
We say that $G$ is \emph{2-step nilpotent} if the quotient $G / Z(G)$ is abelian.
In particular, the Heisenberg groups $\HH_{n}(\K)$, as well as their direct products, are 2-step nilpotent~\cite[Examples~13.36]{dructu2018geometric}.
Every finitely generated 2-step nilpotent group can be embedded as a subgroup of the direct product $A \times G_0$, where $A$ is finite and $G_0$ is a 2-step nilpotent subgroup of $\UT(n, \Q)$ for some $n$~\cite{Baumslag2007LectureNO}.

\paragraph*{Logarithm of matrices and Lie algebra}

The \emph{Lie algebra} $\mun$ is defined as the $\Q$-linear space of $n \times n$ upper triangular rational matrices with \emph{zeros} on the diagonal.
There exist the \emph{logarithm} map
\[
    \log: \UT(n, \Q) \rightarrow \mun, \quad
    A \mapsto \sum_{k = 1}^n \frac{(-1)^{k-1}}{k} (A - I)^k
\]
and the \emph{exponential} map
\[
    \exp: \mun \rightarrow \UT(n, \Q), \quad
    X \mapsto \sum_{k = 0}^n \frac{1}{k!} X^k
\]
which are inverse of one another.
In particular, $\log I = 0$ and $\exp(0) = I$.

The Lie algebra $\mun$ is equipped with the \emph{Lie bracket} $[\cdot, \cdot] : \mun \times \mun \rightarrow \mun$ given by $[X, Y] = XY - YX$.
For a subset or subsemigroup $\mathcal{A}$ of $\UT(n, \Q)$, we naturally denote by $\log \mathcal{A} \coloneqq \{\log a \mid a \in \mathcal{A}\}$ the set of logarithm of matrices in $\mathcal{A}$.

\paragraph*{Parikh Image and length two subwords}
Given a finite alphabet $\mG = \{A_1, \ldots, A_K\}$,
the \emph{Parikh Image} of a word $w = M_1 \cdots M_m$ over the alphabet $\mG$ is the vector $\PI^{\mG}(w) \coloneqq (\PI^{\mG}_1(w), \ldots, \PI^{\mG}_K(w)) \in \Zp^K$, where $\PI^{\mG}_i(w)$ is the number of times $A_i$ appears in $w$.
That is, $\PI^{\mG}_i(w) \coloneqq \card(\{j \mid M_j = A_i\})$.
When the alphabet $\mG$ is clear from the context, we sometimes write $\PI(w), \PI_i(w)$ instead of $\PI^{\mG}(w), \PI^{\mG}_i(w)$.

For $1 \leq i < j \leq K$, let $w$ be a word over the alphabet $\mG$, denote by $\delta^{\mG}_{ij}(w)$ the number of occurrences of the subword $\cdots A_i \cdots A_j \cdots$ minus the number of occurrences of the subword $\cdots A_j \cdots A_i \cdots$ in $w$.
That is, writing $w = M_1 M_2 \cdots M_s$, we have
\[
    \delta_{ij}^{\mG}(w) \coloneqq \delta_{ij}^{\mG,+}(w) - \delta_{ij}^{\mG,-}(w),
\]
where
\begin{align*}
    & \delta_{ij}^{\mG,+}(w) \coloneqq \{(u, v) \mid 1 \leq u < v \leq s, M_u = A_i, M_v = A_j\}, \\
    & \delta_{ij}^{\mG,-}(w) \coloneqq \{(u, v) \mid 1 \leq u < v \leq s, M_u = A_j, M_v = A_i\}
\end{align*}
Again, if the alphabet $\mG$ is clear from the context, we write $\delta_{ij}(w)$ instead of $\delta_{ij}^{\mG}(w)$.
Obviously, we have the parity constraint
\begin{equation}\label{eq:mod2}
\delta_{ij}(w) \equiv \delta_{ij}^{+}(w) + \delta_{ij}^{-}(w) = \PI_i(w) \cdot \PI_j(w) \mod 2.
\end{equation}

\paragraph*{The Baker-Campbell-Hausdorff formula}
Let $G$ be a 2-step nilpotent subgroup of $\UT(n, \Q)$.
The \emph{Baker-Campbell-Hausdorff (BCH) formula}~\cite{baker1905alternants, campbell1897law, hausdorff1906symbolische} states that, given a sequence of matrices $B_1, B_2, \ldots, B_s$ in $G$, we have
\begin{equation}\label{eq:BCHraw}
    \log(B_1 B_2 \cdots B_m) = \sum_{i=1}^m \log B_i + \frac{1}{2}\sum_{1 \leq i < j \leq s} [\log B_i, \log B_j].
\end{equation}

Fix a finite alphabet $\mG = \{A_1, \ldots, A_K\}$ in $G$.
For an arbitrary word $w$ with Parikh Image $\bl = (\ell_1, \ldots, \ell_K)$, applying Equation~\eqref{eq:BCHraw} to the sequence of matrices in $w$ yields
\begin{equation}\label{eq:BCHcomb}
\log w = \sum_{i = 1}^K \ell_{i} \log A_i + \frac{1}{2} \sum_{1 \leq i < j \leq K} \delta_{ij}(w) [\log A_i, \log A_j].
\end{equation}
Here, $\log w$ is understood to be the result of multiplying all matrices appearing in $w$ in order, then taking the logarithm. 
We will adopt this notation throughout this paper.

\section{A combinatorial problem for length two subwords}

First let us describe the general strategy for solving intersection-type decision problems.
Consider a simple example: given two alphabets $\mG = \{A_1, \ldots, A_K\}$, $\mH = \{B_1, \ldots, B_M\}$ in a 2-step nilpotent subgroup of $\UT(n, \Q)$, we want to decide whether $\langle \mG \rangle \cap \langle \mH \rangle \neq \emptyset$.
This boils down to finding two words $v, w$ respectively in the alphabet $\mG$ and $\mH$, such that $\log v = \log w$.
Denote by $\bx = (x_1, \ldots, x_K)$ the Parikh Image of $v$, and by $\by = (y_1, \ldots, y_M)$ the Parikh Image of $w$, then the BCH formula~\eqref{eq:BCHcomb} yields the equivalence between $\log v = \log w$ and
\begin{multline}\label{eq:strict}
    \sum_{i = 1}^K x_{i} \log A_i + \sum_{i < j} \frac{\delta^{\mG}_{ij}(v)}{2} [\log A_i, \log A_j] =
    \sum_{i = 1}^M y_{i} \log B_i + \sum_{i < j} \frac{\delta^{\mH}_{ij}(w)}{2} [\log B_i, \log B_j], \\
    \bx \in \Zp^K, \; \by \in \Zp^M, \quad \PI^{\mG}(v) = \bx, \; \PI^{\mH}(w) = \by. \quad 
\end{multline}
Hence, deciding whether $\langle \mG \rangle \cap \langle \mH \rangle \neq \emptyset$ boils down to solving Equation~\eqref{eq:strict} in the numerical variables $\bx, \by$ and the word variables $v , w$ over alphabets $\mG, \mH$.

Consider a ``relaxed'' version of this problem.
That is, we replace $\delta^{\mG}_{ij}(v)$ and $\delta^{\mH}_{ij}(w)$ by new variables $c_{ij}, d_{ij}$ over integers, without imposing any constraint.
This gives the equation
\begin{multline}\label{eq:relax}
    \sum_{i = 1}^K x_{i} \log A_i + \sum_{1 \leq i < j \leq K} \frac{c_{ij}}{2} [\log A_i, \log A_j] = \sum_{i = 1}^M y_{i} \log B_i + \sum_{1 \leq i < j \leq M} \frac{d_{ij}}{2} [\log B_i, \log B_j], \\
    \bx \in \Zp^K, \; \by \in \Zp^M, \quad c_{ij}, d_{ij} \in \Z \text{ for all } i, j. \quad
\end{multline}
Obviously, if Equation~\eqref{eq:strict} has a solution, then the relaxed version~\eqref{eq:relax} will also admit a solution.
The converse is not necessarily true.
The implicit constraints imposed by the word combinatorial variables $\delta_{ij}^{\mG}(v), \delta_{ij}^{\mH}(w)$ in Equation~\eqref{eq:strict} are highly non-trivial.
(For example, one should at least have $|\delta_{ij}^{\mG}(v)| \leq x_i x_j$ for all $i, j$).
However, these constraints are not reflected by the numerical variables $c_{ij}$ in Equation~\eqref{eq:relax}.

The key idea of this paper is the following surprising fact.
For the two problems we consider (Semigroup Intersection and Orbit Intersection), it is sufficient to solve the relaxed version of the equation, plus several simple constraints (such as the modulo 2 constraint in Equation~\eqref{eq:mod2}).
In particular, given a ``suitable'' solution to the relaxed Equation~\eqref{eq:relax}, we can always construct a solution to Equation~\eqref{eq:strict}.
\emph{A priori}, the values of $\delta^{\mG}_{ij}(v)$ cannot reach all integers like the free variables $c_{ij}$;
nevertheless, when $x_1, \ldots, x_K$ tend towards infinity, the vector $\left(\delta^{\mG}_{ij}(v)\right)_{1 \leq i < j \leq K}$ can in fact reach every value within a ball of radius size $O(|\bx|^2)$, satisfying modulo 2 constraints.
This will suffice to construct a suitable word $v$, as the quadratic radius will eventually dominate the linear term $\sum_{i = 1}^K x_{i} \log A_i$.

This section aims to formalize this idea.
The main result of this section will be Proposition~\ref{prop:nperm}.
First, we prove a simple case where the alphabet consists of two letters.

\begin{lem}\label{lem:twoperm}
    Given an alphabet $\mG = \{A_i, A_j\}$ and non-negative integers $s_i, s_j \in \Zp$, then for every $C \in \Z$ satisfying
    \begin{equation}\label{eq:condtwoperm}
        |C| \leq s_i s_j \quad \text{and} \quad C \equiv s_i s_j \mod 2,
    \end{equation}
    there exists a permutation $w$ of the word $A_i^{s_i} A_j^{s_j}$ such that $\delta_{ij}(w) = C$.
\end{lem}
\begin{proof}
For an illustration of the proof, see Figure~\ref{fig:2let}.
We start with the word $w = A_i^{s_i} A_j^{s_j}$, which satisfies $\delta_{ij}(w) = s_i s_j$.
We gradually swap pairs of consecutive letters in $w$: each time we replace an occurrence of consecutive $A_i A_j$ with $A_j A_i$.
An occurrence of $A_i A_j$ can always be found unless we have reached the ``final'' permutation $A_j^{s_j} A_i^{s_i}$.
It is easy to see that each swap reduces the value of $\delta_{ij}(w)$ by 2.
Therefore, by swapping consecutive $A_i A_j$ one by one, $\delta_{ij}(w)$ can reach every value between $\delta_{ij}(A_i^{s_i} A_j^{s_j}) = s_i s_j$ and $\delta_{ij}(A_j^{s_j} A_i^{s_i}) = - s_i s_j$ that has the same parity with $s_i s_j$.
This proves the lemma.
\end{proof}

\begin{figure}[h]
    \centering
    \includegraphics[width=0.6\textwidth, keepaspectratio, trim={0cm 0.8cm 0.5cm 0.6cm},clip]{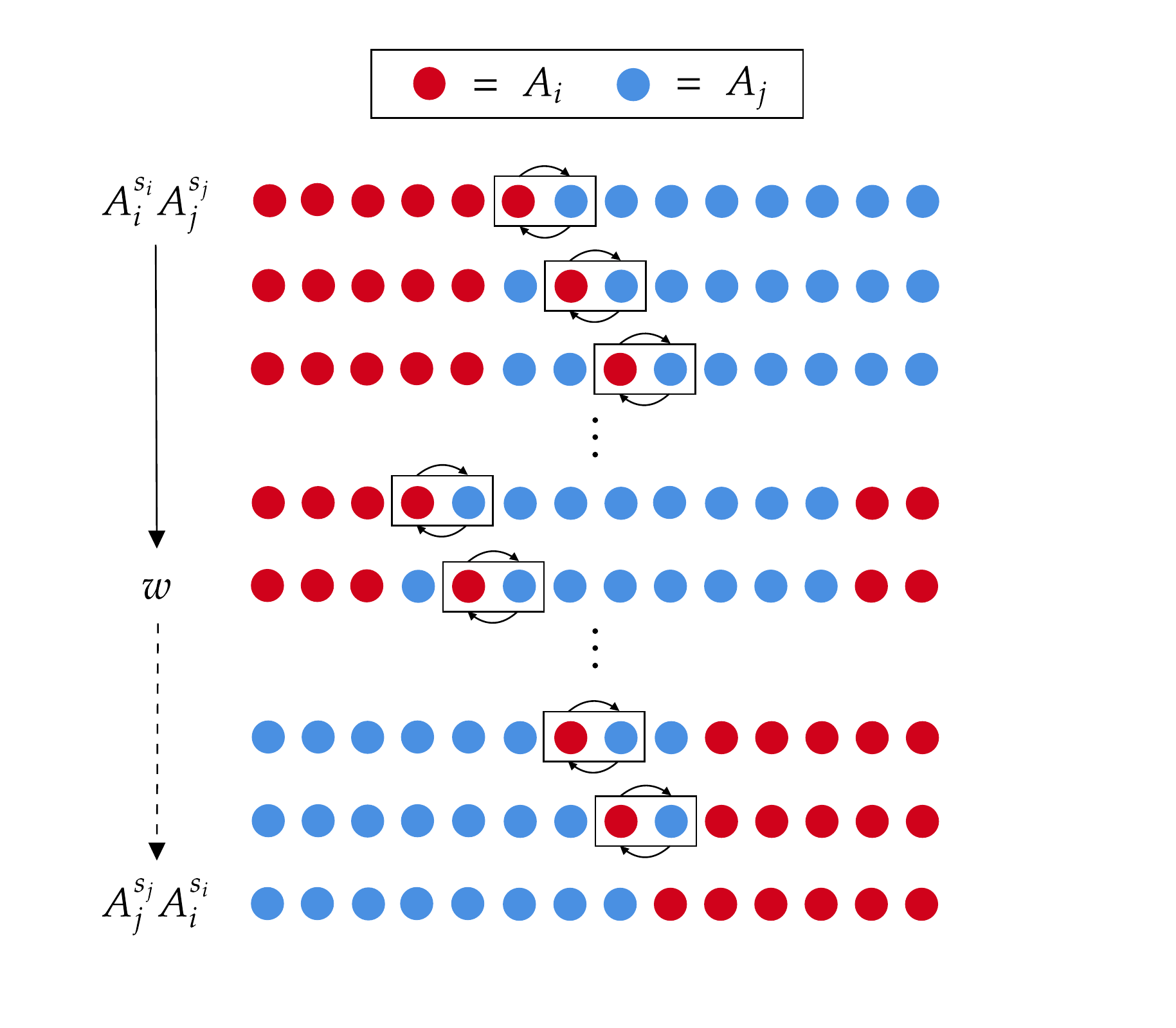}
    \caption{Illustration for the proof of Lemma~\ref{lem:twoperm}.}
    \label{fig:2let}
\end{figure}

\begin{figure}[h]
    \centering
    \includegraphics[width=0.95\textwidth,height=1.0\textheight,keepaspectratio, trim={0.5cm 0.5cm 1cm 1cm},clip]{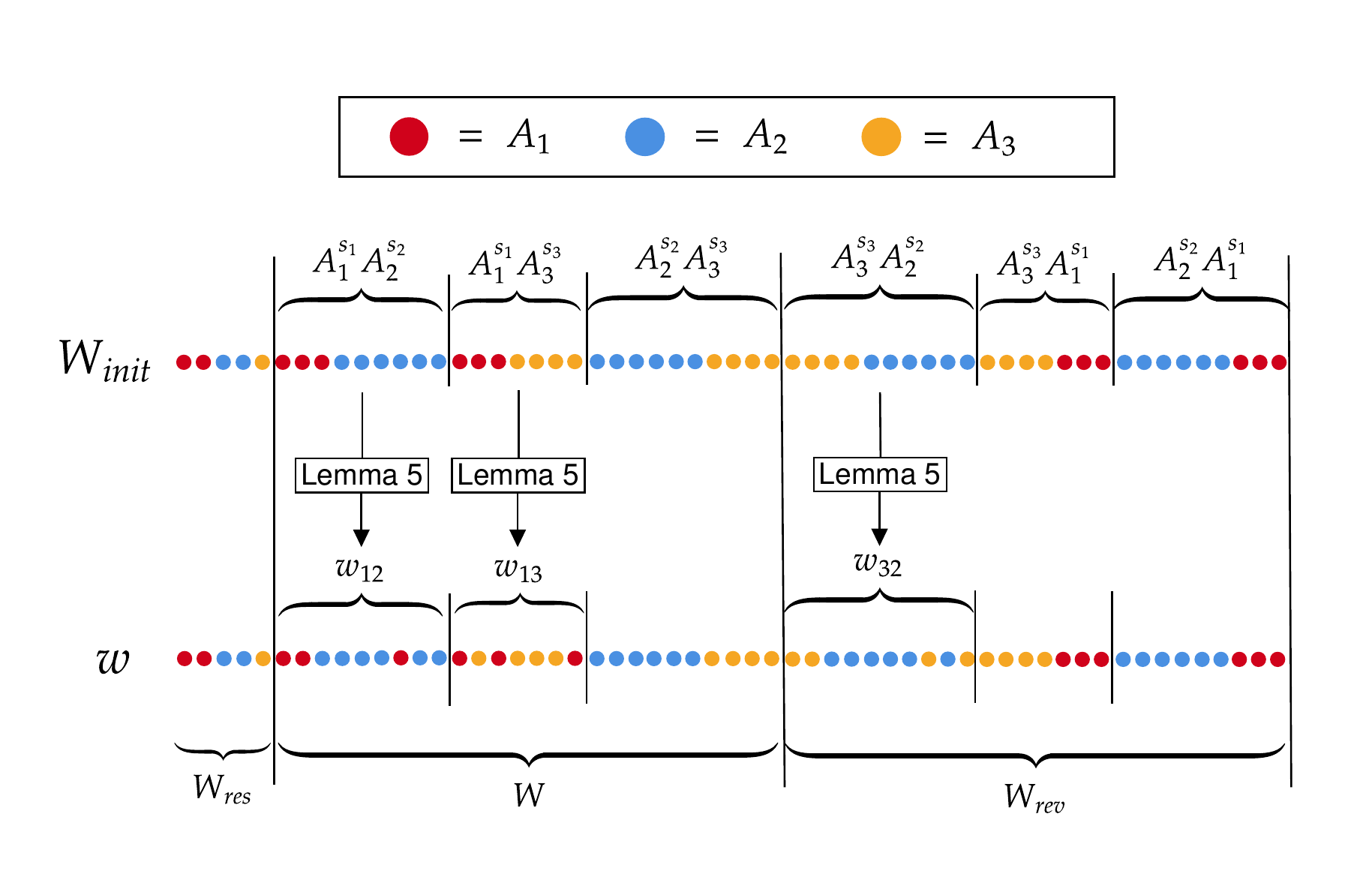}
    \caption{Illustration for the proof of Proposition~\ref{prop:nperm}.}
    \label{fig:morelet}
\end{figure}

We then prove the main result of this section, which generalizes Lemma~\ref{lem:twoperm} to alphabets of more than two letters.

\begin{prop}\label{prop:nperm}
    Fix a finite alphabet $\mG$ of size $K \geq 2$.   
    Then for any tuples $\bl = (\ell_1, \ldots, \ell_K) \in \Zp^K$ and $\{C_{ij}\}_{1 \leq i < j \leq K} \in \Zp^{K(K-1)/2}$ satisfying
    \begin{equation}\label{eq:condCij}
        |C_{ij}| \leq \frac{\ell_i \ell_j}{4K^2} - 2K (\ell_i + \ell_j) - 4K^2, \quad \text{ for all }\; 1 \leq i < j \leq K,
    \end{equation}
    and 
    \begin{equation}\label{eq:condmod}
        C_{ij} \equiv \ell_i \ell_j \mod 2, \quad \text{ for all }\; 1 \leq i < j \leq K,
    \end{equation}
    there exists a word $w$ with Parikh Image $\bl$ such that
    \begin{equation}
        \delta_{ij}(w) = C_{ij}, \quad \text{ for all }\; 1 \leq i < j \leq K.
    \end{equation}
\end{prop}
\begin{proof}
    For an illustration of the proof, see Figure~\ref{fig:morelet}.
    For all $i$, write $\ell_i = 2(K-1) s_{i} + r_{i}$ with $0 \leq r_i < 2(K-1)$.
    Consider the word $W_{init} \coloneqq W_{res} \cdot W \cdot W_{rev}$, where
    \begin{align*}
        W_{res} & \coloneqq A_1^{r_1} A_2^{r_2} \cdots A_K^{r_K}, \\
        W & \coloneqq \left(A_1^{s_1} A_2^{s_2} \right) \left(A_1^{s_1} A_3^{s_3} \right) \cdots \left(A_1^{s_1} A_K^{s_K} \right) \left(A_2^{s_2} A_3^{s_3} \right) \cdots \left(A_2^{s_2} A_K^{s_K} \right) \left(A_3^{s_3} A_4^{s_4} \right) \cdots \left(A_{K-1}^{s_{K-1}} A_K^{s_K} \right), \\
        W_{rev} & \coloneqq \left(A_K^{s_K} A_{K-1}^{s_{K-1}} \right) \left(A_K^{s_K} A_{K-2}^{s_{K-2}} \right) \cdots \left(A_K^{s_K} A_1^{s_1} \right) \left(A_{K-1}^{s_{K-1}} A_{K-2}^{s_{K-2}} \right) \left(A_{K-1}^{s_{K-1}} A_{K-3}^{s_{K-3}} \right) \cdots \left(A_2^{s_2} A_1^{s_1} \right).
    \end{align*}
    In particular, $W$ is the concatenation of all words of the form $A_i^{s_i} A_j^{s_j}$ where $i < j$, and $W_{rev}$ is the reverse of $W$.
    It is easy to verify that $W_{init}$ contains $\ell_i$ occurrences of the letter $A_i$, so its Parikh Image is exactly $\bl$.

    We now compute $\delta_{ij}(W_{init})$ for $i < j$.
    Since $W \cdot W_{rev}$ is a palindrome, we have $\delta_{ij}(W \cdot W_{rev}) = 0$, so
    \begin{align}\label{eq:Winit}
        \delta_{ij}(W_{init}) & = \delta_{ij}(W_{res}) + \PI_i(W_{res}) \PI_j(W \cdot W_{rev}) - \PI_j(W_{res}) \PI_i(W \cdot W_{rev}) \nonumber \\
        & = r_i r_j + r_i \cdot 2(K-1) s_j - r_j \cdot 2(K-1) s_i.
    \end{align}
    In particular, since $0 \leq r_i < 2(K-1)$, we have
    \begin{equation}
        |\delta_{ij}(W_{init})| \leq 4(K-1)^2 + 2(K-1)^2 (s_j + s_i) < 4 K^2 + 2(K-1) (\ell_i + \ell_j)
    \end{equation}

    By Condition~\ref{eq:condCij}, we have
    \begin{align}\label{eq:bounddiff}
        |\delta_{ij}(W_{init}) - C_{ij}| & \leq |\delta_{ij}(W_{init})| + |C_{ij}| \nonumber \\
        & < 4 K^2 + 2(K-1) (\ell_i + \ell_j) + \frac{\ell_i \ell_j}{4K^2} - 2K (\ell_i + \ell_j) - 4K^2 \nonumber \\
        & < \left(\frac{\ell_i}{2(K-1)} - 1\right) \left(\frac{\ell_j}{2(K-1)} - 1\right) \nonumber \\
        & < s_i s_j.
    \end{align}
    Since Equation~\eqref{eq:Winit} yields $\delta_{ij}(W_{init}) \equiv \ell_i \ell_j \mod 2$, Condition~\eqref{eq:condmod} then gives
    \begin{equation}\label{eq:moddiff}
        \delta_{ij}(W_{init}) \equiv C_{ij} \mod 2.
    \end{equation}

    We now show how to construct the word $w$. Starting with the word $W_{init}$, for every pair $i < j$ perform the following:
    \begin{enumerate}
        \item If $\delta_{ij}(W_{init}) > C_{ij}$. By Lemma~\ref{lem:twoperm} there exists a permutation $w_{ij}$ of the word $A_i^{s_i} A_j^{s_j}$ such that
        $
            \delta_{ij}(w_{ij}) = s_i s_j + C_{ij} - \delta_{ij}(W_{init}).
        $
        Indeed, Equations~\eqref{eq:bounddiff} and \eqref{eq:moddiff} guarantee that the conditions~\eqref{eq:condtwoperm} in Lemma~\ref{lem:twoperm} are satisfied.
        We then replace the subword $A_i^{s_i} A_j^{s_j}$ in the $W$-part of $W_{init}$ with the word $w_{ij}$. 
        The resulting new word $W'$ will satisfy
        \[
            \delta_{ij}(W') = \delta_{ij}(W_{init}) - \delta_{ij}(A_i^{s_i} A_j^{s_j}) + \delta_{ij}(w_{ij}) = C_{ij}.
        \]
        This replacement does not change $\delta_{uv}(W_{init})$ for $(u, v) \neq (i, j)$.
        \item If $\delta_{ij}(W_{init}) < C_{ij}$. By Lemma~\ref{lem:twoperm} there exists a permutation $w_{ji}$ of the word $A_j^{s_j} A_i^{s_i}$ such that 
        $
            \delta_{ij}(w_{ji}) = - s_i s_j + C_{ij} - \delta_{ij}(W_{init}).
        $
        Again, Equations~\eqref{eq:bounddiff} and \eqref{eq:moddiff} guarantee that the conditions~\eqref{eq:condtwoperm} in Lemma~\ref{lem:twoperm} are satisfied.
        We then replace the subword $A_j^{s_j} A_i^{s_i}$ in the $W_{rev}$-part of $W_{init}$ with the word $w_{ji}$. 
        The resulting new word $W'$ will satisfy
        \begin{equation*}
            \delta_{ij}(W') = \delta_{ij}(W_{init}) - \delta_{ij}(A_j^{s_j} A_i^{s_i}) + \delta_{ij}(w_{ji}) = C_{ij}.
        \end{equation*}
        This replacement does not change $\delta_{uv}(W_{init})$ for $(u, v) \neq (i, j)$.
        \item If $\delta_{ij}(W_{init}) = C_{ij}$, do not perform any change.
    \end{enumerate}
    Performing all these replacements on $W_{init}$ for all pairs $i < j$ simultaneously, the resulting word $w$ then satisfies $\delta_{ij}(w) = C_{ij}$ for all $1 \leq i < j \leq K$.
\end{proof}

\section{A polynomial time algorithm for Intersection Emptiness}
We prove Theorem~\ref{thm:interempty} in this section.
Let $G$ be a 2-step nilpotent subgroup of $\UT(n, \Q)$.
Let 
\[
    \mG_1 = \{A_{11}, A_{12}, \ldots, A_{1K_1}\}, \ldots, \mG_M = \{A_{M1}, A_{M2}, \ldots, A_{MK_M}\}
\]
be $M$ sets of matrices in $G$. 
The following proposition shows that Intersection Emptiness can be reduced to solving linear Diophantine equations with extra constraints on supports.
The key to obtaining a PTIME algorithm is the fact that these equations are all \emph{homogeneous}.
Hence one can actually solve them over $\Q$, then scale them to obtain integer solutions.

\begin{prop}\label{prop:intertolinalg}
    We have $\langle \mG_1 \rangle \cap \cdots \cap \langle \mG_M \rangle \neq \emptyset$ if and only if there exist non-zero vectors $\bl_1 \in \Zp^{K_1} \setminus \{\bzero\}, \cdots, \bl_M \in \Zp^{K_m} \setminus \{\bzero\}$ as well as rational numbers $c_{mij}$ for $1 \leq m \leq M, i, j \in \supp(\bl_m)$, such that 
    \begin{multline}\label{eq:condinter}
        \sum_{j = 1}^{K_1} \ell_{1j} \log A_{1j} + \sum_{\overset{i < j}{i, j \in \supp(\bl_1)}} c_{1ij} [\log A_{1i}, \log A_{1j}] = \\
        \sum_{j = 1}^{K_2} \ell_{2j} \log A_{2j} + \sum_{\overset{i < j}{i, j \in \supp(\bl_2)}} c_{2ij} [\log A_{2i}, \log A_{2j}] = \cdots \\
        = \sum_{j = 1}^{K_M} \ell_{Mj} \log A_{Mj} + \sum_{\overset{i < j}{i, j \in \supp(\bl_m)}} c_{Mij} [\log A_{Mi}, \log A_{Mj}]        
    \end{multline}
\end{prop}
\begin{proof}
If $\langle \mG_1 \rangle \cap \cdots \cap \langle \mG_M \rangle \neq \emptyset$, let $g$ be an element in the intersection.
There exist non-empty words $w_1, \ldots, w_m$ over the alphabets $\mG_1, \ldots, \mG_M$ such that $\log g = \log w_1 = \cdots = \log w_m$.
By the BCH formula~\eqref{eq:BCHcomb},
\[
\log g = \sum_{j = 1}^{K_i} \PI_i(w_m) \log A_{mj} + \sum_{\overset{i < j}{i, j \in \supp(\bl_m)}} \frac{\delta_{ij}(w_m)}{2} [\log A_{mi}, \log A_{mj}], \quad \text{ for } m = 1, \ldots, M.
\]
This shows that \eqref{eq:condinter} is satisfied by $\bl_m \coloneqq \PI^{\mG_m}(w_m)$ and $c_{mij} \coloneqq \delta_{ij}(w_m)/2$ for $1 \leq m \leq M, i, j \in \supp(\bl_m)$.

For the other implication, suppose such non-zero vectors $\bl_1, \ldots, \bl_M$ and the rational numbers $c_{mij}$ exist.
Then there exists $g \in G$ such that
\begin{equation}\label{eq:g}
\sum_{j = 1}^{K_1} \ell_{mj} \log A_{mj} + \sum_{\overset{i < j}{i, j \in \supp(\bl_m)}} \frac{2c_{mij}}{2} [\log A_{mi}, \log A_{mj}] = \log g, \quad \text{ for } m = 1, \ldots, M.
\end{equation}
Note that if $i, j \in \supp(\bl_m)$ then $\ell_{mi} \ell_{mj} \neq 0$.

By homogeneity, for any $N \in \Zpp$, the vectors $N\bl_1, \ldots, N\bl_M$ and $Nc_{mij}$ also satisfy Condition~\eqref{eq:condinter}.
Hence, multiplying all $\ell_{ij}$, $c_{mij}$ and $\log g$ by a common denominator, we can suppose all $\ell_{ij}$ and $c_{mij}$ to be integers.
Denote $K =  \max_{1 \leq m \leq M} K_m$, then there exists a large enough \emph{even} integer $N \in \Zpp$ such that
\begin{equation}\label{eq:quadN}
    |N \cdot 2 c_{mij}| \leq \frac{N^2 \ell_{mi} \ell_{mj}}{4 K^2} - 2NK(\ell_i + \ell_j) - 4K^2
\end{equation}
for $1 \leq m \leq M, i, j \in \supp(\bl_m)$.
This is because $\ell_{mi} \ell_{mj} > 0$, so the right hand side of \eqref{eq:quadN} is quadratic and dominates the linear term on the left for large enough $N$. 
Replace all $\ell_{ij}$ with $N\ell_{ij}$, all $c_{mij}$ with $N c_{mij}$, and $\log g$ with $N\log g$, then the new variables satisfy $2c_{mij} \equiv 0 \equiv \ell_{mi} \ell_{mj} \mod 2$, and
\begin{equation}
    |2c_{mij}| \leq \frac{\ell_{mi} \ell_{mj}}{4 K^2} - 2K(\ell_i + \ell_j) - 4K^2 \leq \frac{\ell_{mi} \ell_{mj}}{4 K_m^2} - 2K_m(\ell_i + \ell_j) - 4K_m^2
\end{equation}
for all $i, j, m$.
Equation~\eqref{eq:g} is still satisfied after the variable replacements.
Therefore, by Proposition~\ref{prop:nperm}, there exist words $w_1, \ldots, w_M$ over the alphabets $\mG_1, \ldots, \mG_M$ such that $\PI(w_m) = \bl_m$ and $\delta_{ij}(w_m) = 2c_{mij}$ for all $1 \leq m \leq M, i, j \in \supp(\bl_m)$.
These words are non-empty since $\bl_m \neq \bzero$.
Plugging into the BCH formula~\eqref{eq:BCHcomb}, we have
\[
\log w_m = \sum_{j = 1}^{K_m} \ell_{mj} \log A_{mj} + \sum_{\overset{i < j}{i, j \in \supp(\bl_m)}} \frac{2c_{mij}}{2} [\log A_{mi}, \log A_{mj}] = \log g \quad \text{ for } m = 1, \ldots, M.
\]
This shows that $\log g \in \bigcap_{i=1}^M \log \langle \mG_i \rangle \neq \emptyset$.
\end{proof}

Using Proposition~\ref{prop:intertolinalg}, we devise Algorithm~\ref{alg:inter} that decides Intersection Emptiness.

\begin{algorithm}[!ht]
\caption{Algorithm for Intersection Emptiness}
\label{alg:inter}
\begin{description} 
\item[Input:] 
$M$ finite sets of matrices $\mG_1 = \{A_{11}, A_{12}, \ldots, A_{1K_1}\}, \ldots, \mG_M = \{A_{M1}, A_{M2}, \ldots, A_{MK_M}\}$ in the group $G$.
\item[Output:] \textbf{True} (intersection is empty) or \textbf{False} (intersection is not empty).
\end{description}
\begin{enumerate}[Step 1:]
    \item \textbf{Initialization.}
    Set $S_1 \coloneqq \{1, 2, \ldots, K_1\}, \ldots, S_M \coloneqq \{1, 2, \ldots, K_M\}$.
    \item\label{step:intermainloop} \textbf{Main loop.} Repeat the following
    \begin{enumerate} 
        \item Represent the $\Q$-linear subspace of $V \coloneqq \Q^{\sum_{m = 1}^M K_m + \sum_{m = 1}^M \card(S_m)(\card(S_m) - 1)/2}$: 
        \begin{multline}\label{eq:defW}        
            W \coloneqq \bigg\{\big((\ell_{mj})_{1 \leq m \leq M, 1 \leq j \leq K_m}, (c_{mij})_{1 \leq m \leq M, i, j \in S_m} \big) \in V \;\bigg| \\
             \sum_{j = 1}^{K_1} \ell_{1j} \log A_{1j} + \sum_{\overset{i < j}{i, j \in S_1}} c_{1ij} [\log A_{1i}, \log A_{1j}] = \cdots \\
             = \sum_{j = 1}^{K_M} \ell_{Mj} \log A_{Mj} + \sum_{\overset{i < j}{i, j \in S_M}} c_{Mij} [\log A_{Mi}, \log A_{Mj}] \bigg\}
        \end{multline}
        as the solution set of homogeneous linear equations.
        \item Compute the projection of $W$ onto the coordinates $(\ell_{mj})_{1 \leq m \leq M, 1 \leq j \leq K_m}$:
        \begin{multline}\label{eq:defpiW}            
            \pi_{\bl}(W) \coloneqq \bigg\{(\ell_{mj})_{1 \leq m \leq M, 1 \leq j \leq K_m} \in \Q^{\sum_{m = 1}^M K_m} \;\bigg|\; \exists (c_{mij})_{1 \leq m \leq M, i, j \in S_m}, \\
             \big((\ell_{mj})_{1 \leq m \leq M, 1 \leq j \leq K_m}, (c_{mij})_{1 \leq m \leq M, i, j \in S_m} \big) \in W \bigg\}
        \end{multline}
        represented as the solution set of homogeneous linear equations.
        \item Define $\Lambda \coloneqq \Zp^{\sum_{m = 1}^M K_m} \cap \pi_{\bl}(W)$ and compute $\supp(\Lambda)$ using Lemma~\ref{lem:compsupp}.
        \item If $\supp(\Lambda) \cap S_m = S_m$ for all $1 \leq m \leq M$, terminate the loop and go to Step~\ref{step:interoutput}.
        Otherwise, let $S_m \coloneqq \supp(\Lambda) \cap S_m$ for every $m$, and continue with Step~\ref{step:intermainloop}.
    \end{enumerate}
    \item\label{step:interoutput} \textbf{Output.} 
    \vspace{-4mm}
    \begin{enumerate} 
    \item If $S_m = \emptyset$ for any $1 \leq m \leq M$, return \textbf{True}.
    \item Otherwise return \textbf{False}.
    \end{enumerate}
\end{enumerate}
\end{algorithm}

\thminterempty*
\begin{proof}
    Theorem~\ref{thm:interempty} follows from the correctness and polynomial time complexity of Algorithm~\ref{alg:inter}.
    Their proofs are given in Appendix~\ref{app:proofs}, Proposition~\ref{prop:alg1}.
\end{proof}

\section{Decidability of Orbit Intersection in $\HH_3(\Q)$}\label{sec:shiftinter}
We prove Theorem~\ref{thm:shiftinter} in this section.
Let $\mG$ and $\mH$ be finite sets of matrices in the group $\HH_3(\Q)$, and $T, S$ be matrices in $\HH_3(\Q)$.
Our goal is to decide whether $T \cdot \langle \mG \rangle \cap S \cdot \langle \mH \rangle = \emptyset$.
Multiplying both $T \cdot \langle \mG \rangle$ and $S \cdot \langle \mH \rangle$ on the left by $T^{-1}$, one can without loss of generality suppose $T = I$.
That is, it suffices to consider the problem of deciding whether $\langle \mG \rangle \cap S \cdot \langle \mH \rangle = \emptyset$.
Denote by $\varphi \colon \log \HH_3(\Q) \rightarrow \Q^2$ the projection onto the superdiagonal, and by $\pi \colon \log \HH_3(\Q) \rightarrow \Q$ the projection onto the upper right entry:
\[
\varphi \colon
\begin{pmatrix}
    0 & a & c \\
    0 & 0 & b \\
    0 & 0 & 0
\end{pmatrix}
\mapsto (a, b);
\quad \quad
\pi \colon
\begin{pmatrix}
    0 & a & c \\
    0 & 0 & b \\
    0 & 0 & 0
\end{pmatrix}
\mapsto c.\]
One easily verifies that for matrices $X, Y \in \HH_3(\Q)$, we have $[\log X, \log Y] = 0$ if and only if $\varphi(\log X)$ and $\varphi(\log Y)$ are linearly dependant. 
Define the cones
\[
    \mC_{\mG} \coloneqq \langle \varphi(\log \mG) \rangle_{\Qp}, \quad
    \mC_{\mH} \coloneqq \langle \varphi(\log \mH) \rangle_{\Qp}.
\]

\subsection{Easy case: The cone $\mC_{\mG} \cap \mC_{\mH}$ has dimension zero or one}\label{subsec:easy}

The situation in this case is similar to the one discussed in~\cite[Section~3, Case I]{colcombet2019reachability}.

\begin{restatable}{prop}{propeasy}\label{prop:easy}
    Suppose the cone $\mC_{\mG} \cap \mC_{\mH}$ has dimension zero or one.
    Deciding whether $\langle \mG \rangle \cap S \cdot \langle \mH \rangle \neq \emptyset$ can be done by solving finitely many linear Diophantine equations.
\end{restatable}

\subsection{Hard case: The cone $\mC_{\mG} \cap \mC_{\mH}$ has dimension two}\label{subsec:hard}

We have $\langle \mG \rangle \cap S \cdot \langle \mH \rangle \neq \emptyset$ if and only if there exist words $v$ in the alphabet $\mG$ and $w$ in the alphabet $\mH$ such that $\log v = \log Sw$.
Let $\bx = (x_1, \ldots, x_K)$ be the Parikh Image of $v$, and $\by = (y_1, \ldots, y_M)$ be the Parikh Image of $w$.
By the BCH formula~\eqref{eq:BCHraw} and \eqref{eq:BCHcomb}, $\log v = \log S w$ is equivalent to
\begin{multline}\label{eq:main}
    \sum_{i=1}^K x_i \log A_i + \frac{1}{2} \sum_{1 \leq i < j \leq K} \delta^{\mG}_{ij}(v) [\log A_i, \log A_j] = \\
    \log S + \sum_{i=1}^M y_i (\log B_i + \frac{1}{2}[\log S, \log B_i]) + \frac{1}{2}\sum_{1 \leq i < j \leq M} \delta^{\mH}_{ij}(w) [\log B_i, \log B_j]
\end{multline}
The following proposition shows that it suffices to solve a relaxed version of Equation~\eqref{eq:main}.

\begin{prop}\label{prop:shifttolinalg}
    Suppose the cone $\mC_{\mG} \cap \mC_{\mH}$ has dimension two.
    We have $\langle \mG \rangle \cap S \cdot \langle \mH \rangle \neq \emptyset$ if and only if there exists integers $x_i, 1 \leq i \leq K$ and $y_j , 1 \leq j \leq M$ and $c_{ij} , 1 \leq i < j \leq K$ and $d_{ij}, 1 \leq i < j \leq M$, satisfying
    \begin{equation}\label{eq:condshiftlin}
        \sum_{i=1}^K x_i \varphi(\log A_i) = \varphi(\log S) + \sum_{i=1}^M y_i \varphi(\log B_i),
    \end{equation}    
    \begin{multline}\label{eq:condshiftquad}
        \sum_{i=1}^K x_i \pi(\log A_i) + \frac{1}{2} \sum_{1 \leq i < j \leq K} c_{ij} \pi([\log A_i, \log A_j]) = \\
        \pi(\log S) + \sum_{i=1}^M y_i \pi(\log B_i + \frac{1}{2}[\log S, \log B_i]) + \frac{1}{2}\sum_{1 \leq i < j \leq M} d_{ij} \pi([\log B_i, \log B_j])
    \end{multline}
    and
    \begin{equation}\label{eq:condshiftmod}
        c_{ij} \equiv x_i x_j \mod 2, \quad 1 \leq i < j \leq K; \quad\quad
        d_{ij} \equiv y_i y_j \mod 2, \quad 1 \leq i < j \leq M.
    \end{equation}
\end{prop}
\begin{proof}
    If $\langle \mG \rangle \cap s \cdot \langle \mH \rangle \neq \emptyset$, then let $v, w$ be non-empty words over the respectively alphabets $\mG$ and $\mH$, such that $\log v = \log S w$.
    Let $c_{ij} \coloneqq \delta^{\mG}_{ij}(v)$ and $d_{ij} \coloneqq \delta^{\mH}_{ij}(w)$ for all $i, j$.
    Since Equation~\eqref{eq:main} is satisfied, projecting it under $\varphi$ and $\pi$ gives respectively \eqref{eq:condshiftlin} and \eqref{eq:condshiftquad}.
    The parity condition~\eqref{eq:condshiftmod} is obviously due to Equation~\eqref{eq:mod2}.
    Hence we have found the integers $x_i, y_j, c_{ij}, d_{ij}$ satisfying Equations~\eqref{eq:condshiftlin}, \eqref{eq:condshiftquad} and \eqref{eq:condshiftmod}.

    On the other hand, let $x_i, y_j, c_{ij}, d_{ij}$ be integers that satisfy Equations~\eqref{eq:condshiftlin}, \eqref{eq:condshiftquad} and \eqref{eq:condshiftmod}.
    Since $\mC_{\mG}$ and $\mC_{\mH}$ have dimension two, the commutators $[\log A_i, \log A_j]$ and $[\log B_i, \log B_j]$ are not all zero (since $\varphi(A_i)$ are not all linearly dependant, same for $\varphi(B_i)$).
    Hence, there exist integers $C_{ij}, D_{ij}$ such that
    \[
    D \coloneqq \sum_{1 \leq i < j \leq K} C_{ij} \pi([\log A_i, \log A_j]) + \sum_{1 \leq i < j \leq M} D_{ij} \pi([\log B_i, \log B_j]) \in \Qpp.
    \]
    Denote by $E$ the common denominator of all the entries of the matrices $\log A_i$, $\log B_i$, $\log S$, $\frac{1}{2}[\log S, \log B_i]$, $\frac{1}{2}[\log A_i, \log A_j]$ and $\frac{1}{2}[\log B_i, \log B_j]$.
    In particular, $DE$ is an integer.
    
    Since the cone $\mC_{\mG} \cap \mC_{\mH}$ has dimension two, there exist \emph{strictly positive} integers $X_1, \ldots, X_K$ and $Y_1, \ldots, Y_M$, such that
    \begin{equation}\label{eq:XY}
        \sum_{i=1}^K X_i \varphi(\log A_i) = \sum_{i=1}^M Y_i \varphi(\log B_i).
    \end{equation}
    This is because, taking $\bv$ to be a vector in the \emph{interior} of $\mC_{\mG} \cap \mC_{\mH}$ (i.e.\ $\bv$ admits an open neighbourhood contained in $\mC_{\mG} \cap \mC_{\mH}$), then $\bv$ is in the interior of both $\mC_{\mG}$ and $\mC_{\mH}$.
    Hence, there exist strictly positive rational numbers $X'_1, \ldots, X'_K$ and $Y'_1, \ldots, Y'_M$, such that 
    \[
        \sum_{i=1}^K X'_i \varphi(\log A_i) = \bv = \sum_{i=1}^M Y'_i \varphi(\log B_i).
    \]
    Multiplying $X'_1, \ldots, X'_K$ and $Y'_1, \ldots, Y'_M$ by their common denominator gives positive integers satisfying Equation~\eqref{eq:XY}.
    
    For any $N \in \Zpp$, the integers $x_i, y_i, c_{ij}, d_{ij}$ can be replaced by the integers 
    \begin{align*}
        x'_i & \coloneqq x_i + 2 N DE X_i \\
        y'_i & \coloneqq y_i + 2 N DE Y_i \\
        c'_{ij} & \coloneqq c_{ij} - 4 NE C_{ij} \left(\sum_{k = 1}^K X_k \pi(\log A_k) - \sum_{k = 1}^M Y_k \pi(\log B_k + \frac{1}{2}[\log S, \log B_k])\right) \\
        d'_{ij} & \coloneqq d_{ij} + 4 NE D_{ij} \left(\sum_{k = 1}^K X_k \pi(\log A_k) - \sum_{k = 1}^M Y_k \pi(\log B_k + \frac{1}{2}[\log S, \log B_k])\right)
    \end{align*}
    for all $i, j$, while still satisfying Equations~\eqref{eq:condshiftlin}, \eqref{eq:condshiftquad} and \eqref{eq:condshiftmod}.
    Furthermore, when $N$ is large enough, we have
    \begin{align}\label{eq:boundx}
        &x'_i > 0, y'_j > 0, \quad 1 \leq i \leq K, 1 \leq j \leq M, \\
    \label{eq:boundc}
        &|c'_{ij}| \leq \frac{x'_i x'_j}{4K^2} - 2K (x'_i + x'_j) - 4K^2, \quad 1 \leq i < j \leq K,
    \end{align}
    and
    \begin{equation}\label{eq:boundd}
        |d'_{ij}| \leq \frac{y'_i y'_j}{4M^2} - 2M (y'_i + y'_j) - 4M^2, \quad 1 \leq i < j \leq M.
    \end{equation}
    This is because the right hand sides of the inequalities~\eqref{eq:boundc} and \eqref{eq:boundd} are quadratic in $N$, whereas the left hand sides grow linearly in $N$.

    Fix an $N$ such that the inequalities~\eqref{eq:boundx}, \eqref{eq:boundc} and \eqref{eq:boundd} are satisfied.
    Then, by Proposition~\ref{prop:nperm}, there exist non-empty words $v, w$ over the alphabets $\mG$ and $\mH$, such that
    \begin{align*}
        \PI^{\mG}(v) & = (x'_1, \ldots, x'_K), \quad
        \delta^{\mG}_{ij}(v) = c'_{ij}, \quad \text{ for } 1 \leq i < j \leq K, \\
        \PI^{\mH}(w) & = (y'_1, \ldots, y'_K), \quad
        \delta^{\mH}_{ij}(v) = d'_{ij}, \quad \text{ for } 1 \leq i < j \leq M.
    \end{align*}
    (Note that Condition~\eqref{eq:condmod} is guaranteed by Equation~\eqref{eq:condshiftmod}.)
    For these words $v, w$, we have
    \begin{equation*}
        \varphi(\log v) = \sum_{i=1}^K x'_i \varphi(\log A_i) = \varphi(\log S) + \sum_{i=1}^M y'_i \varphi(\log B_i) = \varphi(\log S w),
    \end{equation*}  
    as well as
    \begin{multline*}
        \pi(\log v) = \sum_{i=1}^K x'_i \pi(\log A_i) + \frac{1}{2} \sum_{1 \leq i < j \leq K} c'_{ij} \pi([\log A_i, \log A_j]) = \\
        \pi(\log S) + \sum_{i=1}^M y'_i \pi(\log B_i + \frac{1}{2}[\log S, \log B_i]) + \frac{1}{2}\sum_{1 \leq i < j \leq M} d'_{ij} \pi([\log B_i, \log B_j]) = \pi(\log S w).
    \end{multline*}    
    This shows $\log v = \log S w$, hence $\langle \mG \rangle \cap s \cdot \langle \mH \rangle \neq \emptyset$.
\end{proof}

Combining the two cases in Subsections~\ref{subsec:easy} and \ref{subsec:hard}, we are able to solve the Orbit Intersection problem for $\HH_3(\Q)$.

\thmshiftinter*
\begin{proof}
    See Appendix~\ref{app:proofs}.
\end{proof}

\bibliography{intersection}

\appendix
\section{Omitted proofs and remarks}\label{app:proofs}

{\renewcommand\footnote[1]{}\corheisnilp*}
\begin{proof}
    (i) By the remark in Section~\ref{sec:prelim}, the Heisenberg group $\HH_n(\K)$ can be embedded as a subgroup of the group $\UT(n', \Q)$ for some $n'$, such that the input size only changes at most polynomially.
    A direct product of Heisenberg groups $\HH_{n_1}(\K_1) \times \cdots \times \HH_{n_s}(\K_s)$ can hence be embedded as a subgroup of some direct product $\UT(n'_1, \Q) \times \cdots \times \UT(n'_s, \Q)$, which is itself a subgroup of $\UT(n'_1 + \cdots + n'_s, \Q)$.
    Again, the input size only changes polynomially during these embeddings.
    The Heisenberg groups $\HH_n(\K)$ as well as their direct products are 2-step nilpotent~\cite[Examples~13.36]{dructu2018geometric}, and the property of being 2-step nilpotent is preserved under isomorphism.
    Therefore, Theorem~\ref{thm:interempty} shows that Intersection Emptiness for $\HH_n(\K)$ as well as for their direct products is decidable in PTIME.

    (ii) Given a finite presentation or a consistent polycyclic presentation of $G$, an embedding $\phi: G \hookrightarrow A \times G_0$ where $A$ is finite and $G_0$ is a 2-step nilpotent subgroup of some $\UT(n, \Q)$ can be effectively computed (see proof of \cite[Corollary~1.8]{https://doi.org/10.48550/arxiv.2208.02164}).
    Denote by $\pi_0: A \times G_0 \rightarrow G_0$ the projection onto $G_0$.
    
    We claim that $\langle \mG_1 \rangle \cap \cdots \cap \langle \mG_M \rangle \neq \emptyset$ if and only if $\langle \pi_0(\phi( \mG_1 ))\rangle \cap \cdots \cap \langle \pi_0(\phi(\mG_M ))\rangle \neq \emptyset$.
    Suppose $g \in \langle \mG_1 \rangle \cap \cdots \cap \langle \mG_M \rangle$, then obviously $\pi_0(\phi(g)) \in \langle \pi_0(\phi( \mG_1 ))\rangle \cap \cdots \cap \langle \pi_0(\phi(\mG_M ))\rangle$.
    On the other hand, suppose $h \in \langle \pi_0(\phi( \mG_1 ))\rangle \cap \cdots \cap \langle \pi_0(\phi(\mG_M ))\rangle = \pi_0(\phi(\langle \mG_1 \rangle)) \cap \cdots \cap \pi_0(\phi(\langle \mG_M \rangle))$, then there exist $a_1, \ldots, a_M \in A$, such that $(a_i, h) \in \langle \mG_i \rangle$ for all $i$.
    Then $(1, h^{\card(A)}) = (a_i, h)^{\card(A)} \in \langle \mG_i \rangle$ for all $i$, hence $\langle \mG_1 \rangle \cap \cdots \cap \langle \mG_M \rangle \neq \emptyset$.

    Since $G_0$ is a 2-step nilpotent subgroup of $\UT(n, \Q)$, one can decide whether $\langle \pi_0(\phi( \mG_1 ))\rangle \cap \cdots \cap \langle \pi_0(\phi(\mG_M ))\rangle = \emptyset$ by Theorem~\ref{thm:interempty}.
    Thus, we conclude that it is decidable whether $\langle \mG_1 \rangle \cap \cdots \cap \langle \mG_M \rangle \neq \emptyset$.
\end{proof}

We did not attempt to analyse the exact complexity of deciding Intersection Emptiness for arbitrary finitely generated 2-step nilpotent groups.
This is because this complexity depends on the computation and representation of the embedding $\phi$, as well as the size of the finite group $A$.

\thmshiftinter*
\begin{proof}
As mentioned in the beginning of Section~\ref{sec:shiftinter}, one can without loss of generality suppose $T = I$, and decide whether $\langle \mG \rangle \cap S \cdot \langle \mH \rangle \neq \emptyset$.
Given $\mG$ and $\mH$, one can effectively compute $\mC_{\mG} \cap \mC_{\mH}$ and its dimension using linear programming~\cite{schrijver1998theory}.

If $\mC_{\mG} \cap \mC_{\mH}$ has dimension zero or one, then Proposition~\ref{prop:easy} shows we can decide whether $\langle \mG \rangle \cap S \cdot \langle \mH \rangle \neq \emptyset$ by solving a finite number of linear Diophantine equations of the form~\eqref{eq:easylin}.

If $\mC_{\mG} \cap \mC_{\mH}$ has dimension two, then Proposition~\ref{prop:shifttolinalg} shows we can decide whether $\langle \mG \rangle \cap S \cdot \langle \mH \rangle \neq \emptyset$ by solving Equations~\eqref{eq:condshiftlin}, \eqref{eq:condshiftquad} and~\eqref{eq:condshiftmod}.
Equation~\eqref{eq:condshiftmod} can be replaced by a boolean combination of conditions of the form ``$x_i \equiv 0 \mod 2$'', ``$x_i \equiv 1 \mod 2$'', ``$y_i \equiv 0 \mod 2$'', $\ldots$, or ``$d_{ij} \equiv 1 \mod 2$''.
Each of these conditions can be expressed as a linear equation over integers, for example ``$x_i \equiv 1 \mod 2$'' is equivalent to ``$x_i = 2 x'_i + 1, x'_i \in \Z$''.
Therefore, solving Equations~\eqref{eq:condshiftlin}, \eqref{eq:condshiftquad} and~\eqref{eq:condshiftmod} is equivalent to solving a boolean combination of linear equations over integers, which is decidable by integer programming.
\end{proof}

\propeasy*
\begin{proof}
    Let $\mL \subseteq \Q^2$ be a linear space of dimension one that contains $\mC_{\mG} \cap \mC_{\mH}$.
    Then we decompose $\mG$ and $\mH$ into disjoint subsets: $\mG = \mG_0 \cup \mG_+$, $\mH = \mH_0 \cup \mH_+$, where
    \begin{align*}
    \mG_0 & \coloneqq \{ A_i \in \mG \mid \varphi(\log A_i) \in \mL \}, \quad \mG_+ \coloneqq \mG \setminus \mG_0; \\
    \mH_0 & \coloneqq \{ B_i \in \mH \mid \varphi(\log B_i) \in \mL \}, \quad \mH_+ \coloneqq \mH \setminus \mH_0.
    \end{align*}
    The key observation is that all matrices in $\mG_0$ and in $\mH_0$ commute with each other (all $\varphi(\log A_i)$ and $\varphi(\log B_j)$ are linearly dependant, so $[\log A_i, \log A_j] = [\log B_i, \log B_j] = 0$).
    
    Suppose $\langle \mG \rangle \cap S \cdot \langle \mH \rangle \neq \emptyset$,
    that is, there exist words $v$ in the alphabet $\mG$ and $w$ in the alphabet $\mH$ such that $\log v = \log S w$.
    We show that the number of occurrences of letters of $\mG_+$ in $v$ is bounded; similarly, the number of occurrences of letters of $\mH_+$ in $w$ is bounded.
    
    Let $\bn$ be a non-zero vector orthogonal to $\mL$, then $\bx \mapsto \bn^{\top} \bx$ is the projection parallel to $\mL$.
    Since $\mC_{\mG} \cap \mC_{\mH} \subseteq \mL$, the values $\bn^{\top} \varphi(\log A_i), A_i \in \mG$ have signs opposite to that of $\bn^{\top} \varphi(\log B_j), B_j \in \mH$.
    Without loss of generality, suppose $\bn^{\top} \varphi(\log A_i) \geq 0$ for all $A_i \in \mG$ and $\bn^{\top} \varphi(\log B_j) \leq 0$ for all $B_j \in \mH$.
    Since $\bn$ is orthogonal to $\mL$, we have furthermore $\bn^{\top} \varphi(\log A_i) > 0$ for all $A_i \in \mG_+$ and $\bn^{\top} \varphi(\log B_j) < 0$ for all $B_j \in \mH_+$; as well as $\bn^{\top} \varphi(\log X) = 0$ for all $X \in \mG_0 \cup \mH_0$.
    
    Now, $\log v = \log S w$ yields $\varphi(\log v) = \varphi(\log S) + \varphi(\log w)$.
    Projecting onto $\bn$, this shows
    \[
    \sum_{i, A_i \in \mG_+} \PI^{\mG}_i(v) \cdot \bn^{\top} \varphi(\log A_i) = \bn^{\top} \varphi(S) + \sum_{i, B_i \in \mH_+} \PI^{\mH}_i(w) \cdot \bn^{\top} \varphi(\log B_i).
    \]
    This yields
    \begin{equation}\label{eq:boundPI}
    \PI^{\mG}_i(v) \leq \frac{\bn^{\top} \varphi(\log S)}{\bn^{\top} \varphi(\log A_i)},
    \quad 
    \PI^{\mH}_j(v) \leq \frac{\bn^{\top} \varphi(\log S)}{\bn^{\top} \varphi(\log B_j)},
    \end{equation}
    for all $A_i \in \mG_+$ and $B_j \in \mH_+$.
    This gives bounds $\beta_{\mG} \coloneqq \sum_{i, A_i \in \mG_+} \frac{\bn^{\top} \varphi(\log S)}{\bn^{\top} \varphi(\log A_i)}$ and $\beta_{\mH} \coloneqq \sum_{i, B_i \in \mH_+} \frac{\bn^{\top} \varphi(\log S)}{\bn^{\top} \varphi(\log B_i)}$, such that if $\log v = \log S w$, then the number of letters of $\mG_+$ in $v$ is bounded by $\beta_{\mG}$; and similarly the number of letters of $\mH_+$ in $w$ is bounded by $\beta_{\mH}$.
    
    Write $v = v_0 C_1 v_1 C_2 \cdots v_{s-1} C_{s} v_s$, where $C_1, \ldots, C_s$ are matrices in $\mG_+$, and $v_0, \ldots, v_s$ are words in the alphabet $\mG_0$.
    Similarly, write $w = w_0 D_1 w_1 D_2 \cdots w_{t-1} D_{t} w_t$, where $D_1, \ldots, D_t$ are matrices in $\mH_+$, and $w_0, \ldots, w_t$ are words in the alphabet $\mH_0$.
    Write $\mG_0 = \{A'_1, \ldots, A'_{K'}\}$ and $\mH_0 = \{B'_1, \ldots, B'_{M'}\}$.
    Define $x_{ij} \coloneqq \PI^{\mG_0}_j (v_i)$ for $0 \leq i \leq s, 1 \leq j \leq K'$, and $y_{ij} \coloneqq \PI^{\mH_0}_j (w_i)$ for $0 \leq i \leq t, 1 \leq j \leq M'$.
    Then $\log v = \log S w$ is equivalent to
    \begin{multline}\label{eq:easylin}
        \sum_{i = 1}^s \log C_i + \sum_{i = 0}^s \sum_{j = 1}^{K'} x_{ij} \log A'_j + \frac{1}{2}\sum_{0 \leq i < k \leq s} \sum_{j = 1}^{K'} x_{ij} [\log A'_j, \log C_k] \\
        + \frac{1}{2}\sum_{1 \leq k \leq i \leq s} \sum_{j = 1}^{K'} x_{ij} [\log C_k, \log A'_j] \\
        = \log S + \sum_{i = 1}^t (\log D_i + \frac{1}{2}[\log S, \log D_i]) + \frac{1}{2}\sum_{i = 0}^t \sum_{j = 1}^{M'} y_{ij} [\log S, \log B'_j] \\
        + \frac{1}{2}\sum_{0 \leq i < k \leq t} \sum_{j = 1}^{M'} y_{ij} [\log B'_j, \log D_k] + \frac{1}{2}\sum_{1 \leq k \leq i \leq t} \sum_{j = 1}^{M'} y_{ij} [\log D_k, \log B'_j]
    \end{multline}
    All other terms contain $[\log A'_i, \log A'_j]$ or $[\log B'_i, \log B'_j]$ and hence vanish by the commutativity of $\mG_0$ and $\mH_0$.
    Note that Equation~\eqref{eq:easylin} is a linear Diophantine equation in the variables $x_{ij}, y_{ij}$.
    Therefore, $\log v = \log S w$ has a solution if and only if there exist matrices $C_1, \ldots, C_s$ in $\mG_+$ and matrices $D_1, \ldots, D_t$ in $\mH_+$, such that Equation~\eqref{eq:easylin} has a solution in non-negative integers, with the additional constraint that, if $s = 0$, then $(x_{01}, \ldots, x_{0K'}) \neq \bzero$; and if $t = 0$, then $(y_{01}, \ldots, y_{0M'}) \neq \bzero$.
    This additional constraint comes from the condition that $v, w$ are not empty words.
    Recall the bounds $s \leq \beta_{\mG}$ and $t \leq \beta_{\mH}$.
    Hence, deciding whether $\log v = \log S w$ has a solution amounts to solving finitely many linear Diophantine equations of the form~\eqref{eq:easylin}.
\end{proof}

In theory, it is possible to give a bound on the complexity of the procedure described in Proposition~\ref{prop:easy}.
The size of the each bound in Equation~\eqref{eq:boundPI} is exponential in the bit size of the entries $S, \mG, \mH$.
Hence the procedure consists of solving exponentially many linear Diophantine equations.

\begin{prop}\label{prop:alg1}
    Algorithm~\ref{alg:inter} is correct and terminates in polynomial time.
\end{prop}
\begin{proof}
    We prove that Algorithm~\ref{alg:inter} outputs \textbf{False} if and only if $\langle \mG_1 \rangle \cap \cdots \cap \langle \mG_M \rangle \neq \emptyset$.
    
    After each iteration of Step~\ref{step:intermainloop}, $\card(S_1) + \cdots + \card(S_M)$ strictly decreases. 
    Therefore, the algorithm terminates after at most $K_1 + \cdots + K_M$ iterations of Step~\ref{step:intermainloop}.
    
    We now show correctness of the algorithm.
    We first show that when Algorithm~\ref{alg:inter} returns \textbf{False}, then $\bigcap_{i=1}^M \langle \mG_i \rangle \neq \emptyset$.
    Suppose the algorithm terminates with output False, the condition in Step~\ref{step:intermainloop}(d) shows that $\supp(\Lambda) \cap S_m = S_m$ for all $1 \leq m \leq M$.
    By the additivity of $\Lambda$ (that is, $\ba, \bb \in \Lambda \implies \ba+\bb \in \Lambda$), there exists a vector $\bl = (\bl_1, \ldots, \bl_M) \in \Lambda$ such that $\supp(\bl) = \supp(\Lambda)$.
    This yields $\supp(\bl_m) = \supp(\Lambda) \cap S_m = S_m$ for all $m$.
    Since $\supp(\bl_m) = S_m \neq \emptyset$, we have $\bl_m \neq \bzero$ for all $1 \leq m \leq M$.
    By the definition~\eqref{eq:defpiW} of $\pi_{\bl}(W)$, there exist rational numbers $(c_{mij})_{1 \leq m \leq M, i, j \in S_m}$ such that 
    \begin{multline}\label{eq:alleqS}
        \sum_{j = 1}^{K_1} \ell_{1j} \log A_{1j} + \sum_{\overset{i < j}{i, j \in S_1}} c_{1ij} [\log A_{1i}, \log A_{1j}] = \sum_{j = 1}^{K_2} \ell_{2j} \log A_{2j} + \sum_{\overset{i < j}{i, j \in S_2}} c_{2ij} [\log A_{2i}, \log A_{2j}] \\
        = \cdots = \sum_{j = 1}^{K_M} \ell_{Mj} \log A_{Mj} + \sum_{\overset{i < j}{i, j \in S_M}} c_{Mij} [\log A_{Mi}, \log A_{Mj}]        
    \end{multline}
    Since $S_m = \supp(\bl_m)$ for all $m$, Equation~\eqref{eq:alleqS} is identical to Equation~\eqref{eq:condinter} in Proposition~\ref{prop:intertolinalg}.
    Therefore Proposition~\ref{prop:intertolinalg} shows $\bigcap_{i=1}^M \langle \mG_i \rangle \neq \emptyset$.

    Next, we show that if $\bigcap_{i=1}^M \langle \mG_i \rangle \neq \emptyset$, then Algorithm~\ref{alg:inter} returns \textbf{False}.
    Suppose $\bigcap_{i=1}^M \langle \mG_i \rangle \neq \emptyset$.
    By Proposition~\ref{prop:intertolinalg}, there exist $\bl_1 = (\ell_{1j})_{1 \leq j \leq K_1} \in \Zp^{K_1} \setminus \{\bzero\}, \ldots, \bl_M = (\ell_{Mj})_{1 \leq j \leq K_M} \in \Zp^{K_M} \setminus \{\bzero\}$, and rational numbers $(c_{mij})_{1 \leq m \leq M, i, j \in \supp(\bl_m)}$ that satisfies Equation~\eqref{eq:condinter} in Proposition~\ref{prop:intertolinalg}.
    We show that ``$\supp(\bl_m) \subseteq S_m$ for all $1 \leq m \leq M$'' is an invariant of the algorithm.

    At initialization, we obviously have $\supp(\bl_m) \subseteq S_m = \{1, \ldots, K_m\}$.
    Before each iteration of \ref{step:intermainloop}(d), suppose we have $\supp(\bl_m) \subseteq S_m$ for all $m$, then Equation~\eqref{eq:condinter} shows that $(\ell_{mj})_{1 \leq m \leq M, 1 \leq j \leq K_m} \in \pi_{\bl}(W)$.
    Consequently, $\supp(\bl_m) \subseteq \supp(\Lambda)$, meaning $\supp(\bl_m) \subseteq S_m$ still holds after \ref{step:intermainloop}(d).

    This invariant shows that $\supp(\bl_m) \subseteq S_m$ for all $m$ by the start of Step~\ref{step:interoutput}.
    Since $\bl_m \in \Zp^{K_m} \setminus \{\bzero\}$, $\supp(\bl_m)$ is non-empty for every $m$.
    We conclude that $S_m \neq \emptyset$ for all $m$ by the start of Step~\ref{step:interoutput}.
    Therefore, Algorithm~\ref{alg:inter} returns \textbf{False}.

    Finally, we show that Algorithm~\ref{alg:inter} terminates in polynomial time.
    Recall that the algorithm terminates after at most $K_1 + \cdots + K_M$ iterations of Step~\ref{step:intermainloop}.
    At each iteration of Step~\ref{step:intermainloop}(b), the projection can be computed in polynomial time by eliminating the variables $(c_{mij})_{1 \leq m \leq M, i, j \in S_m}$ from the equations defining $W$. 
    Then, at each iteration of Step~\ref{step:intermainloop}(c) the support $\supp(\Lambda)$ is computed by Lemma~\ref{lem:compsupp}. 
    The total input size of the linear programming instances is polynomial with respect to the total bit length of the matrix entries in $\mG_1, \ldots, \mG_M$.
    Indeed, the total bit length of $\log A_{mi}$ and $[\log A_{mi}, \log A_{mj}]$ is at most of quadratic size in $\mG_m$; and the projection performed in Step~\ref{step:intermainloop}(b) can only alter the total entry bit size at most polynomially.
    From this, one can express $\pi_{\bl}(W)$ as the solution set of a system of homogeneous linear equations whose total bit length is polynomial in $\mG_1, \ldots, \mG_M$.
    Hence Lemma~\ref{lem:compsupp} computes the support of $\Lambda \coloneqq \Zp^{\sum_{m = 1}^M K_m} \cap \pi_{\bl}(W)$ in polynomial time.
    Therefore, each iteration of Step~\ref{step:intermainloop} takes polynomial time, and thus the overall complexity of Algorithm~\ref{alg:inter} is polynomial with respect to the input $\mG_1, \ldots, \mG_M$.
\end{proof}

\end{document}